\documentclass[a4wide,10pt,reqno]{article}
\usepackage{amsfonts}
\usepackage{amsmath}
\usepackage{amsthm}
\usepackage{amssymb}
\usepackage{lipsum}
\usepackage{mathrsfs}
\usepackage{hyperref}
\usepackage{enumerate}
\usepackage{array}
\usepackage{tabularx}
\usepackage{float}
\restylefloat{table}
\usepackage{amssymb,tikz}
\usepackage{caption}
\usepackage{subcaption}
\usepackage{amsmath}
\usepackage{amsthm}
\usepackage{amsbsy}
\usepackage{psfrag}
\usepackage{pstricks}
\usepackage{xcolor}
\usepackage{graphics}
\usepackage{graphicx}
\usepackage{epstopdf,bm}
\numberwithin{equation}{section}

\usepackage{graphicx}
\usepackage{enumitem}

\theoremstyle{plain}
\newtheorem{theorem}{Theorem}[section]
\newtheorem{defi}[theorem]{Definition}
\newtheorem{remark}[theorem]{Remark}

\newtheorem{lemma}[theorem]{Lemma}

\def\O{\Omega}

\def\cV{\mathcal{V}}

\def\cT{\mathcal{T}}
\def\cE{\mathcal{E}}

\def\mean#1{\left\{\hskip -5pt\left\{#1\right\}\hskip -5pt\right\}}
\def\jump#1{\left[\hskip -3.5pt\left[#1\right]\hskip -3.5pt\right]}
\def\smean#1{\{\hskip -3pt\{#1\}\hskip -3pt\}}
\def\sjump#1{[\hskip -1.5pt[#1]\hskip -1.5pt]}

\usepackage{fancyhdr}
\title{DBCP}
\author{sudiptoc506 }
\date{August 2023}

\begin{document}
	\title{Non-Conforming Finite Element Method For Constrained Dirichlet Boundary Control Problem}
	\author{ Sudipto Chowdhury
		\thanks{ Department of Mathematics, The LNM Institute of Information Technology Jaipur, Rajasthan- 302031, India (\tt{sudipto.choudhary@lnmiit.ac.in})} ~, ~~
		Divay Garg\thanks{Department of Mathematics, Indian Institute of Technology Delhi,
			New Delhi- 110016, India (\tt{divaygarg2@gmail.com})[The work of the author is granted by CSIR Research Fellowship] } }
	\maketitle
	
	\begin{abstract}
		\noindent
		This article examines the Dirichlet boundary control problem governed by the Poisson equation, where the control variables are square integrable functions defined on the boundary of a two dimensional bounded, convex, polygonal domain. It employs an ultra weak formulation and utilizes Crouzeix-Raviart finite elements to discretize the state variable, while employing piecewise constants for the control variable discretization. The study demonstrates that the energy norm of an enriched discrete optimal control is uniformly bounded with respect to the discretization parameter. Furthermore, it establishes an optimal order \emph{a priori} error estimate for the control variable. 
	\end{abstract}
	\par
	\noindent
	{\small{\bf Keywords:}}
	Optimal control problem, Crouziex-Raviart, Finite element method, A priori error analysis, Control constraints \\
	\noindent
	{\small{\bf AMS subject classification:} 65N30, 65N15, 65N12, 65K10 }
	\everymath{\displaystyle}
\section{Introduction}

We consider the control constrained Dirichlet boundary control problem on a convex and bounded polygonal domain $\O\subseteq \mathbb{R}^2$. We denote the polygonal boundary of the domain $\O$ by $\Gamma:=\cup_{j=1}^{k}\Gamma_{j}$, where $\Gamma_{j}$ is a straight line segment for all $1\leq j\leq k$. The minimization problem seeks a pair $(\bar{y},\bar{u})\in L_2(\O)\times U_{ad}$ such that
\begin{align}
J(\bar{y},\bar{u})=\min_{(y,u)\in L_2(\O)\times U_{ad}} J(y,u),
\end{align}
subject to the following second order elliptic equation:
\begin{align}\label{eqn:state}
\begin{cases}
-\Delta y&=f \quad \text{in} \quad \O,\\
\hspace{.4cm}y&=u \quad \text{on} \quad \Gamma.
\end{cases}
\end{align}
The given function $f\in L_2(\O)$ denotes the force which acts on the system externally and $U_{ad} \subseteq L_2(\Gamma)$ is the set of admissible controls which will be defined later. The quadratic cost functional 
 $J(\cdot,\cdot):L_2(\O)\times L_2(\Gamma)\rightarrow \mathbb{R}$ is defined in the following manner:
\begin{align}\label{cost:functional}
J(y,u):=\frac{1}{2}\int_{\O}|y-y_d|^2\,dx+\frac{\alpha}{2}\int_{\Gamma}|u|^2\,ds,
\end{align}
where $y_d\in L_2(\O)$ is the given desired state and $\alpha>0$ is known parameter which is introduced for the regularization. 
\noindent
The state equation \eqref{eqn:state} under the standard weak formulation is not well posed for $u\in L_2(\Gamma)$ as the weak solution of \eqref{eqn:state} lie in $H^1(\O)$ space and therefore by trace theorem \cite{brennerbook}, the first trace of it must lie in $H^{1/2}(\Gamma)$ which is not the case in this formulation. There are multiple techniques available in the literature to overcome this variational difficulty, one of which is to use the method of transposition which is also known as the ultra weak formulation \cite{casasraymond:2006,may:2013}. In brevity, we discuss the method of transposition. Define a test function space as follows:
\begin{align*}
X:=\{\phi \in H^1_0(\O):\Delta \phi \in L_2(\O)\}.
\end{align*}
By elliptic regularity theory on convex polygonal domains, we have $X=H^2(\O)\cap H^1_0(\O)$.
Now, we obtain the ultra weak formulation of the state equation \eqref{eqn:state}. Multiplying $\phi\in X$ in both sides of \eqref{eqn:state} and then integrate both the sides to find 
\begin{align*}
-\int_{\O}\Delta y\,\phi\,dx&=\int_{\O}f\,\phi\,dx.
\end{align*}
A use of integration by parts in the above equation yields 
\begin{align*}
\int_{\O}f\,\phi\,dx&=\int_{\O}\nabla y\cdot\nabla \phi\,dx-\int_{\Gamma}\phi\,\frac{\partial y}{\partial n}\,ds\\
&=\int_{\O}\nabla y\cdot\nabla \phi\,dx,
\end{align*}
where we have used the fact that $\phi \in H^1_0(\O)$ to get the last equality. The idea in the method of transposition is to apply the integration by parts once more. By doing so, we get
\begin{align*}
\int_{\O}f\,\phi\,dx&=-\int_{\O}y\,\Delta \phi\,dx+\int_{\Gamma}y\,\frac{\partial \phi}{\partial n}\,ds.
\end{align*}
Thus, the ultra weak formulation of \eqref{eqn:state} can be stated as: Given $f\in L_2(\O)$ and $u\in L_2(\Gamma)$, find $y\in L_2(\O)$ such that
\begin{align}\label{ultra:weak}
(y,\Delta \phi)&=-(f,\phi)+\langle u,\partial \phi/\partial n\rangle \quad \forall ~ \phi \in X,
\end{align}
where $(\cdot,\cdot)$ and $\langle \cdot,\cdot\rangle$ denotes the $L_2$ inner product on the domain $\O$ and its boundary $\Gamma$ respectively. Another approach to study second order Dirichlet boundary control problem is to 
transform the Dirichlet boundary condition into Robin's type boundary condition \cite{casasmateosraymond:2009}. Other method is to employ the energy space based approach in which the control is penalized in the energy space $H^{1/2}(\Gamma)$ \cite{chowdhurykkt,phan:2015}.  In this article the control variable is discretized with piecewise constant finite elements while the state and adjoint states are discretized using first-order non conforming Crouzeix Raviart finite elements.

In the finite element analysis for the Dirichlet boundary control problems addressed so far in the literature polynomial degrees used to discretize the state variable or the adjoint state variable and the control variable happens to be the same. But in this article we have discretized the control variable with piecewise constant finite elements whereas the state and adjoint variables are discretized using Crouzeix Raviart finite elements. Therefore the main challenge is to associate the discrete state variables with the discrete control variables via an invertible bounded linear operator such that the operator norm of the inverse remains uniformly bounded with respect to the discretization parameter $h>0$ (to be defined later).
In order to construct this operator in \eqref{asso opr} a mild restriction on the triangulation is imposed (discussed in Section 4 ).
This constitutes one of the main novelties of this article. Apart from this an enrichment operator is introduced satisfying the orthogonality described in Theorem \ref{Orthog1}, which plays a crucial role in the entire analysis. Moreover in Theorem \ref{lem2}, \ref{lem:1},\ref{supconv8}, \ref{sup-conv} and \ref{supconv2} several super convergence results are derived, which also contributes significantly to the novelty of this article.

\section{Continuous Problem}
In this segment, we discuss the variational setting of the investigated problem and derive the first order necessary optimality conditions. Before, we enter into the analysis part, we define the notations which are used throughout the article. The notations $H^m(\O)$, $H_0^m(\O)$, and $W^{m,p}(\O)$ represent the standard Sobolev spaces on $\O$, where $m$ is a non negative integer.  The Sobolev spaces of non integer order $s$, $H^s(\O)$ and $W^{s,p}(\O)$, are defined using interpolation, as explained in \cite{brennerbook}. On a Lipschitz domain, this definition is equivalent to the definition that uses double-integral norms, which is discussed in \cite[Theorem 14.2.3]{brennerbook}. The norms of $H^s(\Gamma)$ and $W^{s,p}(\Gamma)$, $0 \leq s \leq 1$ and $1 < p < \infty$, are defined on the boundary $\Gamma$ using charts, which is equivalent to using double-integral norms on $\Gamma$, according to \cite{brennerbook,Nezaetal}. The outwards unit normal vector to $\Gamma$ is denoted by $n$. The dual space of $H^{1/2}(\Gamma)$ is denoted by $H^{-1/2}(\Gamma)$ equipped with the operator norm $\|\cdot\|_{H^{-1/2}(\Gamma)}$ defined as:
\begin{align*}
\|\mu\|_{H^{-1/2}(\Gamma)}:=\sup\limits_{v\in H^{1/2}(\Gamma), v\neq 0}\frac{\langle \mu,v\rangle}{\|v\|_{H^{1/2}(\Gamma)}}, \quad \mu \in H^{-1/2}(\Gamma).
\end{align*}
Any generic member of $\mathbb{R}^2$ is denoted by $x=(x_1,x_2)$, where $x_1$, $x_2\in \mathbb{R}$. In the article, $C$ denotes a positive constant, which is generic in nature and does not depend on the solutions or the mesh-size.
\begin{remark}
When $v$ belongs to $H^2(\O)$, the trace of its gradient, denoted by $\nabla v|_{\Gamma}$, lies in $H^{1/2}(\Gamma)^2$. If $\Gamma$ has a smooth boundary, the outwards unit normal vector $n$ is continuous, allowing the definition of the normal derivative $\partial_n v = n\cdot\nabla v$ to be well defined. Thus $\partial_n v \in H^{1/2}(\Gamma)$ for $v \in H^2(\O)$ and the following estimate holds:
\begin{align*}
\|\partial _n v\|_{H^{1/2}(\Gamma)}\leq C\|v\|_{H^2(\O)}, \quad v\in H^2(\O).
\end{align*}
However, this estimate is not applicable for polygonal boundaries $\Gamma$, which are only Lipschitz continuous. Nonetheless, for $v \in H^2(\O)$, we still have $\partial_n v|_{\Gamma_j} \in H^{1/2}(\Gamma_j)$ for each straight component $\Gamma_j$, $1\leq j \leq k$ of $\Gamma$. To accommodate this, we introduce the space $\tilde{H}^{1/2}(\Gamma):=\{v \in L_2(\Gamma), v|_{\Gamma_j} \in H^{1/2}(\Gamma_j) \quad \forall~ 1\leq j \leq m\}$. We denote by $\tilde{H}^{-1/2}(\Gamma)$ the completion of $L_2(\Gamma)$ with respect to the following dual norm on $L_2(\Gamma)$
\begin{align*}
\|\mu\|_{\tilde{H}^{-1/2}(\Gamma)}:=\sup\limits_{v\in X\setminus\{0\}} \frac{\langle \mu,\partial_n v\rangle}{\|v\|_{H^2(\O)}}\leq \sup_{v\in \tilde{H}^{1/2}(\Gamma),v\neq 0}\frac{\langle \mu,v\rangle}{\|v\|_{\tilde{H}^{1/2}(\Gamma)}}.
\end{align*}
Note that $\tilde{H}^{-1/2}(\Gamma)$ is not in general the dual space of $\tilde{H}^{1/2}(\Gamma)$. In the case of a smooth boundary $\Gamma$, 
the mapping $\partial_n: X \rightarrow H^{1/2}(\Gamma)$ is surjective, so $\tilde{H}^{-1/2}(\Gamma)= H^{-1/2}(\Gamma)$.
\end{remark}
\noindent
For any $u\in \tilde{H}^{-1/2}(\Gamma)$, the harmonic extension of $u$ is denoted by $y_u\in L_2(\O)$ which satisfies the following equation:
\begin{align}\label{harmonic:extn}
(y_u,\Delta \phi)&=\langle u,\partial _n \phi\rangle \quad \forall ~ \phi \in X.
\end{align}
The function $u\in \tilde{H}^{-1/2}(\Gamma)$ is known as the ultra weak trace of $y_u\in L_2(\O)$ and the following trace estimate holds true:
\begin{align*}
\|u\|_{\tilde{H}^{-1/2}(\Gamma)}=\sup\limits_{\phi\in X\setminus\{0\}} \frac{(y_u,\Delta \phi)}{\|\phi\|_{H^2(\O)}}\lesssim \|y_u\|_{L_2(\O)}.
\end{align*}
Furthermore, if $u\in H^{1/2}(\Gamma)$, then $y_u\in H^1(\O)$ becomes the standard harmonic extension of $u$ such that it satisfies the following equation:
\begin{align*}
(\nabla y_u,\nabla \phi)=0 \quad \forall ~ \phi \in V.
\end{align*}
In this case $u\in H^{1/2}(\Gamma)$ is simply the trace of $y_u$ in the sense that $y_u|_{\Gamma}=u$.
\subsection{Optimality System}
In this section, we write the model Dirichlet boundary optimal control problem and derive the corresponding optimality system by applying the first order necessary optimality conditions. For $(f,u)\in L_2(\O)\times L_2(\Gamma)$, define $(y_f,y_u)\in L_2(\O)\times L_2(\O)$ to be the solutions of the following equations:
\begin{align}
(y_f,\Delta\phi)&=-(f,\phi)\quad \forall ~ \phi \in X,\label{vol:eqn}\\
(y_u,\Delta\phi)&=\langle u, \partial_n \phi\rangle \quad \forall ~ \phi \in X.\label{bdry:eqn}
\end{align}
In view of \eqref{vol:eqn}-\eqref{bdry:eqn}, the ultra weak solution $y\in L_2(\O)$ of \eqref{ultra:weak} can be written as $y=(y_f+y_u)$. In addition, if the given data satisfy $(f,u)\in L_2(\O)\times H^{1/2}(\Gamma)$, then we can recover the standard weak formulation, i.e. the following holds:
\begin{align}\label{eqn:weakstate}
\begin{cases}
y&=y_f+y_u,~ y_f\in H^1_0(\O),~ y_u\in H^1(\O),\\
y_u&=u \quad\text{on}~ \Gamma, \\
a(y_f,v)&=(f,v)-a(y_u,v)\quad \forall ~v\in H^1_0(\O),
\end{cases}
\end{align}
where $a(\cdot,\cdot):H^1(\O)\times H^1(\O)\rightarrow \mathbb{R}$ is a symmetric bilinear form given by:
\begin{align*}
a(v,w)&=\int_{\O}\nabla v\cdot \nabla w\,dx \quad \forall ~ v, w\in H^1(\O).
\end{align*}
The bilinear form $a(\cdot,\cdot)$ defines a continuous mapping on $H^1(\O)$ and coercive map on $H^1_0(\O)$. Set $V=H^1_0(\O)$ and $Q=H^1(\O)$. Given that $\O$ is a convex polygonal domain, for $u = 0$, the weak solution $y$ of \eqref{eqn:state} belongs to $H^2(\O)\cap H^1_0(\O)$. Additionally, it satisfies the following \emph{a priori} bound.
\begin{align*}
\|y\|_{H^2(\O)}\leq C \|f\|_{L_2(\O)}.
\end{align*}
The subsequent lemma affirms the well-posed-ness of the boundary value problem \eqref{eqn:state} in its ultra weak form, and as a particular instance, it ensures the presence of the ultra weak harmonic extension of the general boundary data $u \in \tilde{H}^{-1/2}(\Gamma)$.
\begin{lemma}\label{lem:existence}
For any given $u \in \tilde{H}^{-1/2}(\Gamma)$ and $f\in H^{-2}(\O)$, the state equation in its ultra weak form \eqref{ultra:weak} has a unique solution $y\in L_2(\O)$ such that the following \emph{a priori} estimate is satisfied:
\begin{align}\label{est:priori}
\|y\|_{L_2(\O)}\leq C\Big( \|f\|_{H^{-2}(\O)}+\|u\|_{\tilde{H}^{-1/2}(\Gamma)}\Big),
\end{align}
where $H^{-2}(\O)$ is the dual space of $X$.
\begin{proof}
For $f\in H^{-1}(\O)$ and $u \in H^{1/2}(\Gamma)$, there exists a unique weak solution $y\in H^1(\O)$ of \eqref{eqn:state} satisfying \eqref{eqn:weakstate}. An application of integration by parts in \eqref{eqn:weakstate} yields
\begin{align}\label{1.1}
-(y,\Delta \phi)&=(f,\phi)-\langle u,\partial_n\phi\rangle \quad \forall ~ \phi \in X,
\end{align}
which implies that $y$ is the solution of \eqref{ultra:weak}. We apply duality argument to prove the \emph{a priori} error estimate \eqref{est:priori}. Consider $z\in V$ to be the weak solution of the following Poisson problem:
\begin{align*}
\begin{cases}
-\Delta z&=y \quad \text{in} \quad \O,\\
\hspace{.4cm}z&=0 \quad \text{on} \quad \Gamma.
\end{cases}
\end{align*}
By elliptic regularity theory, $z\in H^2(\O)$ and $\|z\|_{H^2(\O)}\leq C \|y\|_{L_2(\O)}$. Thus, using \eqref{1.1} for $\phi=z$, we get
\begin{align*}
\|y\|^2_{L_2(\O)}&=(y,-\Delta z)\\
&=(f,z)-\langle u,\partial _n z\rangle,
\end{align*}
which on using Cauchy-Schwarz inequality yields \eqref{est:priori}. Since $H^{-1}(\O)\subseteq H^{-2}(\O)$ and $H^{1/2}(\Gamma)\subseteq \tilde{H}^{-1/2}(\Gamma)$ are dense, by density arguments, for $u\in \tilde{H}^{-1/2}(\Gamma)$ and $f\in H^{-2}(\O)$, the equation \eqref{ultra:weak} possesses a unique solution $y\in L_2(\O)$ satisfying the \emph{a priori} error estimate \eqref{est:priori}.
\end{proof}
\end{lemma}
\noindent
In light of the Lemma \ref{lem:existence}, we define the control to state map which maps the given data to the unique solution of \eqref{ultra:weak}.
\begin{defi}\label{control:to:state}
For $(f,u)\in L_2(\O)\times L_2(\Gamma)$, the solution map $S: L_2(\O)\times L_2(\Gamma)\rightarrow L_2(\O)$ is defined as $S(f,u)=y=(y_f+y_u)$ such that $y_f,y_u$ satisfy \eqref{vol:eqn} and \eqref{bdry:eqn}, respectively.
\end{defi}
\noindent
Now, we define the space of admissible controls $U_{ad}$ as follows: 
\begin{align*}
U_{ad}:=\{u\in L_2(\Gamma):u_a\leq u(x)\leq u_b ~~ \text{a.e} ~~ x\in \Gamma\},
\end{align*}
where $u_a$, $u_b$ $\in \mathbb{R}$ are such that $u_a<u_b$. The model Dirichlet boundary optimal control problem reads as: find $(\bar{y},\bar{u})\in L_2(\O)\times U_{ad}$ such that
\begin{align}\label{mocp}
J(\bar{y},\bar{u})=\min_{(y,u)\in L_2(\O)\times U_{ad}}J(y,u),
\end{align}
subject to the constraint $y=S(f,u)$, where $J(\cdot,\cdot)$ is the quadratic cost functional defined by \eqref{cost:functional}. Next, we present the following result that addresses the well-posedness of the underlying optimal control problem \eqref{mocp} and the corresponding first-order optimality system.
\begin{theorem}
The optimal control problem \eqref{mocp} possess a unique solution $(\bar{y},\bar{u})\in L_2(\O)\times U_{ad}$. Moreover the unique pair $(\bar{y},\bar{u})$ solves the following first order necessary optimality conditions:
\begin{align}
&\bar{y}=\bar{y}_f+y_{\bar{u}}, \quad \bar{y}_f,~ y_{\bar{u}} \in L_2(\O),\label{ctskkt:1}\\
&(\bar{y}_f,\Delta \phi)=-(f,\phi) \quad \forall ~ \phi \in X,\label{ctskkt:2}\\
&(y_{\bar{u}},\Delta \phi)=\langle \bar{u}, \partial_n \phi\rangle \quad \forall ~ \phi \in X,\label{ctskkt:3}\\
&(\bar{y}-y_d,y_u-y_{\bar{u}})+\alpha\langle\bar{u},u-\bar{u}\rangle\geq 0\quad \forall ~ u \in U_{ad}.\label{ctskkt:4}
\end{align}
\begin{proof}
By utilizing the solution operator $S(\cdot,\cdot)$, we can reformulate the model optimal control problem \eqref{mocp} into a simplified form, expressed as follows:
\begin{align}\label{red:mocp}
\min_{u\in U_{ad}}j(u),
\end{align}
where $j(\cdot):L_2(\Gamma)\rightarrow \mathbb{R}$ represents the reduced cost functional defined as:
\begin{align*}
j(u):=\frac{1}{2}\|S(f,u)-y_d\|^2_{L_2(\O)}+\frac{\alpha}{2}\|u\|^2_{L_2(\Gamma)}.
\end{align*}
The reduced cost functional $j(\cdot)$ defines a strictly convex functional on a non empty closed convex Hilbert space $U_{ad}$. Thus, the standard theory of PDE constrained optimal control problem \cite{trolzbook} implies that problem \eqref{red:mocp} has a unique solution say $\bar{u}\in U_{ad}$. Setting $\bar{y} = S(f, \bar{u})$ as the corresponding state variable. By definition of $S(\cdot,\cdot)$, we find that $\bar{y}\in L_2(\O)$ satisfies \eqref{ctskkt:1}-\eqref{ctskkt:3}. By applying the first order necessary optimality conditions, we get for any $u\in U_{ad}$
\begin{align*}
0\leq  j'(\bar{u})(u-\bar{u})&:=\lim_{t\downarrow 0}\frac{j(\bar{u}+t(u-\bar{u}))-j(\bar{u})}{t}\\
&=(\bar{y}-y_d,y_u-y_{\bar{u}})+\alpha\langle\bar{u},u-\bar{u}\rangle,
\end{align*}
which implies that $\bar{u}\in U_{ad}$ solves \eqref{ctskkt:4}. This completes the proof.
\end{proof}
\end{theorem}
\subsection{Regularity of the Optimal Control} In this fragment, we obtain the regularity of the optimal variables namely $\bar{y}$ and $\bar{u}$ which are useful for the subsequent analysis. Introduce an adjoint state $\theta\in V$ such that it satisfies the following variational formulation:
find $\theta\in V$ such that:
\begin{align}\label{As1}
(\nabla v,\nabla \theta)=(\bar{y}-y_{d},v)\quad \forall~ v\in V.
\end{align}
By elliptic regularity theory for convex polygonal domain \cite{grisvardbook}, $\theta\in X$ and it satisfies the following \emph{a priori} estimate:
\begin{align*}
\|\theta\|_{H^2(\O)}&\leq C \|\bar{y}-y_d\|_{L_2(\O)}.
\end{align*}
Using integration by parts and density arguments in \eqref{As1}, we get
\begin{align}\label{As2}
(v,\Delta \theta)&=-(\bar{y}-y_d,v)\quad \forall~ v\in L_2(\O).
\end{align}
Since $\theta\in H^2(\O)$, it implies $\partial _n \theta \in \tilde{H}^{1/2}(\Gamma)$. For any $u\in U_{ad}$, using \eqref{As2} for $v=y_u-y_{\bar{u}}\in L_2(\O)$, we get
\begin{align*}
(\bar{y}-y_d,y_u-y_{\bar{u}})&=-(y_u-y_{\bar{u}},\Delta \theta)\\
&=-\langle u-\bar{u},\partial_n\theta\rangle,
\end{align*} 
where we have used \eqref{harmonic:extn} to observe the last equality. Hence, the inequality \eqref{ctskkt:4} reduces to 
\begin{align*}
\alpha \langle \bar{u},u-\bar{u}\rangle-\langle u-\bar{u},\partial_n\theta\rangle&\geq 0 \quad \forall ~ u\in U_{ad},
\end{align*}
which implies
\begin{align}\label{ctskkt:5}
\langle \alpha \bar{u}-\partial _n \theta,u-\bar{u}\rangle \geq 0 \quad \forall ~ u\in U_{ad}.
\end{align}
It follows from \eqref{ctskkt:5} and standard arguments from \cite[Chapter 2]{trolzbook} that
\begin{align}\label{ctskkt:6}
\bar{u}(x)=P_{[u_a, u_b]}\Big(\dfrac{1}{\alpha}(\partial _n \theta)(x)\Big),\quad \text{for almost every} ~ x\in \Gamma,
\end{align}
where $P_{[u_a, u_b]}(w):=\min\{u_b,\max\{u_a,w\}\}$ denotes the projection of $\mathbb{R}$ onto $[u_a,u_b]$. Since $\Omega$ is convex, therefore it follows from \eqref{ctskkt:6} that $\bar{u}\in \tilde{H}^{1/2}(\Gamma)$. As $\theta$ equals zero on $\Gamma$, the derivative of $\theta$ in the tangential direction also equals zero on $\Gamma$. Therefore, at the corner points of the domain denoted as $x$, we have $(\partial _n \theta)(x)=0$. Hence, it follows from \eqref{ctskkt:6} that $\bar{u}(x)=0$ at all the corner points $x$ of the domain, which implies that $\bar{u}\in H^{1/2}(\Gamma)$.

\section{Preliminaries }
We introduce the following notations which are used throughout the article.
\begin{itemize}
\item $\cT_h$ is a regular triangulation of $\O$ into closed triangles $T$.
\item $h_T:=$ diam ($T$) is the diameter of the triangle $T\in \cT_h$.
\item $h:=\max_{T\in \cT_h}h_T$ is the mesh parameter.
\item $\cV_h^i:=$ the set of all interior vertices of $\cT_h$.
\item $\cV_h^b:=$ the set of all boundary vertices of $\cT_h$.
\item $\cV_h:=\cV_h^i\cup \cV_h^b$ is the set of all vertices of $\cT_h$.
\item $\cE_h^i:=$ the set of all interior edges of $\cT_h$.
\item $\cE_h^b:=$ the set of all boundary edges of $\cT_h$.
\item $\cE_h:=\cE_h^i\cup \cE_h^b$ is the set of all edges of $\cT_h$.
\item $m_e:=$ mid point of $e\in \cE_h$.
\item $\mathbb{P}_k(T):=$ the set of all polynomials of degree atmost $k\in \mathbb{N}\cup \{0\}$ over $T\in \cT_h$.
\item $H^k(\O,\cT_h):=\{v\in L_2(\O):v|_{T}\in H^k(T) \quad \forall~ T\in \cT_h\}$, $k\in \mathbb{N}\cup \{0\}$.
\end{itemize} 
From this point forward, the notation $a\lesssim b$ indicates the existence of a positive constant $C$ that is not dependent on the mesh parameter $h$, such that a $a\leq Cb$. It is assumed that the triangulation $\cT_h$ is regular, meaning that any two distinct simplices in $\cT_h$ with non-empty intersection are either identical, or share exactly one common vertex, or one common edge. Additionally, every interior angle of any simplex is bounded from below by a universal positive constant $\rho$. All the generic constants hidden in the notation $\lesssim$ are solely determined by $\rho>0$, thereby ensuring the triangulation is shape regular.\\
Next, we will establish the definitions for the jump and average of functions that are scalar-valued and vector-valued. Consider an edge $e \in \cE_h^i$ that is shared by two adjacent triangles $T_{+}$ and $T_{-}$, such that $e = T_{+}\cap T_{-}$. Let $n_{+}$ be the unit normal of $e$ pointing from $T_{+}$ to $T_{-}$, and $n_{-}= -n_{+}$. In this context, we define the jumps $\sjump{\cdot}$ and averages $\smean{\cdot}$ across the edge $e$ as follows:
\begin{itemize}
\item For a scalar valued function $w\in H^1(\O,\cT_h)$, define
\begin{align*}
\sjump{w}:=w_{+}\,n_{+}+w_{-}\,n_{-}, \quad \quad \smean{w}:=\frac{w_{+}+w_{-}}{2},
\end{align*}
where $w_{\pm}=w|_{T_{\pm}}$.
\item For a vector valued function $v\in [H^1(\O,\cT_h)]^2$, define
\begin{align*}
\sjump{v}:=v_{+}\cdot n_{+}+v_{-}\cdot n_{-}, \quad \quad \smean{v}:=\frac{v_{+}+v_{-}}{2},
\end{align*}
where $v_{\pm}=v|_{T_{\pm}}$.
\end{itemize}
To simplify the notation, we also introduce the concepts of jump and mean on the boundary $\Gamma$. Consider any edge $e \in \cE_h^b$, where it is evident that there exists a triangle $T \in \cT_h$ such that $e = \partial T \cap \Gamma$. Let $n_e$ be the unit normal of $e$ pointing outward from $T$. For any $w \in H^1(\O,\cT_h)$ and any $v \in [H^1(\O,\cT_h)]^2$, we define the following on $e \in \cE_h^b$:
\begin{align*}
\sjump{w}&:=w\,n_e \quad \text{and} \quad \smean{w}=w,\\
\sjump{v}&:=v\cdot n_e \quad \text{and} \quad \smean{v}=v.
\end{align*}
Below, we define the Crouziex-Raviart finite element spaces. 
\begin{align*}
V_h&:=\{v_h\in L_2(\O):v_h|_{T}\in \mathbb{P}_1(T) ~ \forall ~ T\in \cT_h ~ \text{and} ~ v_h ~\text{is continuous at}~ m_e ~\forall ~e\in \cE_h \},\\
V_h^0&:=\{v_h\in V_h: v_h(m_e)=0 ~\forall ~ e\in \cE_h^b \}.
\end{align*}
\textbf{We need to modify the definition of $V_h$. The mean integral of the function should be continuous i.e DOF should be modified to mean integrals.}\\
The discrete bilinear form $a_{pw}(\cdot,\cdot):V_h\times V_h\rightarrow \mathbb{R}$ is defined by
\begin{align*}
a_{pw}(v_h,w_h)&=\sum_{T\in \cT_h}\int_{T}\nabla v_h|_T\cdot\nabla w_h|_T\,dx \quad \forall~ v_h,w_h\in V_h. 
\end{align*}
The mesh dependent bilinear form $a_{pw}(\cdot,\cdot)$ defines a symmetric bilinear form which coincides with the continuous bilinear form  $a(\cdot,\cdot)$ on $ H^1(\O)\times H^1(\O)$. Now, we define the mesh dependent norm on $V_h$. For any $v_h\in V_h$, define $\|\cdot\|_h$ as follows:
\begin{align*}
\|v_h\|_{h}:=\sqrt{a_{pw}(v_h,v_h)}.
\end{align*}
Note that, $\|\cdot\|_h$ defines only a semi-norm on $V_h$ but a norm on $V_h^0$. Next, we define a Crouziex-Raviart interpolation map $I_{CR}:V\cap C(\bar{\O})\rightarrow V_h^0$ as follows:
\begin{align}\label{CR;inter}
\int_{e}I_{CR}v\,ds=\int_{e}v\,ds\quad \forall ~ e\in \cE_h, ~ v\in V\cap C(\bar{\O}).
\end{align}
By density arguments, $I_{CR}$ can be extended continuously and uniquely to $V$, and the extension is again denoted by $I_{CR}$ for the ease of presentation. In the next lemma, we state the approximation properties of $I_{CR}$ \cite{CRetal}.
\begin{lemma}\label{lem:CRinterest}
For $v\in V$, it holds that:
\begin{align*}
\|v-I_{CR}v\|_{L_2(\O)}+h\|v-I_{CR}v\|_{h}&\lesssim h|v|_{H^1(\O)}.
\end{align*}
\end{lemma}
\subsection{Raviart-Thomas Interpolation Operator}
In this segment, Raviart-Thomas interpolation operator is defined which is used further for the error analysis. Given a triangle $T\in \cT_h$, the lowest order Raviart-Thomas space is defined as:
\begin{align*}
RT_{0}(T):=\mathbb{P}_0(T)^2+x\,\mathbb{P}_0(T),
\end{align*} 
where $x=(x_1,x_2)\in \mathbb{R}^2$. 
\begin{remark}
Let $T\in \cT_h$ be any triangle. For any $v\in RT_{0}(T)$, we have $v\cdot n_{e}\in \mathbb{P}_{0}(e)$ for all $e\in \partial T$, where $n_{e}$ is the unit outward normal vector to $e$.
\end{remark}
\noindent
For any $T\in \cT_h$, define the local interpolation operator $\Pi_T: H^1(T)^2 \rightarrow RT_{0}(T)$ as follows:
\begin{align*}
\int_{e}\Pi_T v\cdot n_{i}\,p\,ds&=\int_{e}v\cdot n_{i}\,p\,ds \quad \forall ~ p\in \mathbb{P}_{0}(e), \quad \forall ~ e\in \partial T.
\end{align*}
\noindent
We state the approximation property of the interpolation operator $\Pi_T$ in the following lemma. 
\begin{lemma}\label{RTI}
Let $T\in \cT_h$ be any fixed triangle. For any $v\in H^1(T)^2$, the following holds:
\begin{align*}
\|v-\Pi_T v\|_{L_2(T)}&\lesssim h_T|v|_{H^1(T)}.
\end{align*}
\begin{proof}
We refer the readers to the article \cite{RTIntr} for the proof.
\end{proof}
\end{lemma}

%


\subsection{Enrichment Operator}
The focus of this section is on introducing a new enrichment operator which plays a crucial role in our upcoming analysis. It is well known that the enrichment operators play a crucial role in the analysis for non-conforming and discontinuous Galerkin finite element methods \cite{BrennerEnrich1, BrennerEnrich2, BrennerEnrich3, CarstensenEnrich1, CarstensenEnrich2, CarstensenEnrich3, karakashian, MaoEnrich}. These operators connect the non-conforming or discontinuous Galerkin spaces to it's conforming counterpart.\\
The novelty of the enrichment operator being introduced here is due to two facts. The degrees of freedom for the finite element functions of it's conforming counterparts are taken in a non-standard manner which enables us to prove an orthogonality result for the CR-finite element functions (Theorem \ref{Orthog1}). This result plays a crucial role in obtaining the optimal order energy norm error estimate for the solution of the state equation.\\
To begin with we modify the definitions of degrees of freedoms of the conforming counterpart $V_c$ of $V_h$.
\begin{align}
V_c=\{&v_h\in V: v_h|_T\in P_2(T), \; v_h|_+(p)=v_h|_-(p)\;\text{for}\; \;p\in\nu\; \text{and}\; \int_ev_h|_+ds=\int_ev_h|_-ds\;\notag\\ &\text{where}\; v_h|_+=v_h|_{T_+},\;v_h|_-=v_h|_{T_-},\; e\in\mathcal{E}_h\}.
\end{align}

Note that this choice of $V_c$ enables the Lagrange interpolation operator $I_L: H^2(\Omega)\rightarrow V_c$ to have the following orthogonal property:
\begin{align}
\int_e(v-I_Lv)ds=0,
\end{align}
which plays a crucial role to prove Theorem \ref{Orthog1}.


Note that, $V_c$ denotes the Lagrange finite element space of order $2$, which is the conforming counterpart of $V_h^0$. Let $I_{L}: V\cap C(\bar{\Omega})\rightarrow V_{c}$ denotes the Lagrange interpolation operator defined as follows: for $v\in V\cap C(\bar{\Omega})$,
\begin{align}
(I_L v)(p)&=v(p) \quad \forall ~ p \in \cV_h,\notag\\
\int_{e}I_L v\,ds&=\int_{e}v\,ds \quad \forall ~ e\in \cE_h.\label{inter:1}
\end{align}

Let $I_c:V_h^0\rightarrow V_c$ denotes the enriching map defined by: for $v_h\in V_h^0$, define
\begin{align}
I_c v_h(p)&=\frac{1}{|\cT_p|}\sum_{T\in \cT_p}v_h|_{T}(p) \quad \forall ~p\in \cV_h^i,\notag\\
I_c v_h(p)&=0 \quad \forall ~p\in \cV_h^b,\label{enrich:1}\\
\int_{e}I_c v_h\,ds&=\int_{e}v_h\,ds \quad \forall ~e\in \cE_h,\label{enrich:2}
\end{align}
where $\cT_p$ denotes the set of triangles sharing the node $p\in \cV_h$ and $|\cT_p|$ refers to the cardinality of $\cT_p$. 
\noindent
Subsequently, we prove an orthogonality result concerning $I_{c}$, which holds significant importance in the subsequent analysis.
\begin{theorem}\label{Orthog1}
Let $p_{h}$ be a piecewise linear function (need not be in $V_{h}$), it holds that:
	\begin{align*}
	a_{pw}(p_{h},v_{h}-I_cv_{h})=0\quad \forall~ v_{h}\in V_h^0. 
	\end{align*}
	\end{theorem}
\begin{proof} 
Let $v_h\in V_h^0$ be an arbitrary function. It follows from triangle wise integration by parts that
\begin{align}\label{3.5}
		a_{pw}(p_{h},v_{h}-I_cv_{h})=&\sum_{e\in \cE_h}\int_{e}\mean{\dfrac{\partial p_h}{\partial n}}\jump{v_{h}-I_{c}v_{h})}ds \notag \\
		&+\sum_{e\in \cE_h^i}\int_{e}\jump{\dfrac{\partial p_{h}}{\partial n}}\mean{v_{h}-I_{c}v_{h}}ds.
\end{align}
Since, $p_h|_{T}\in \mathbb{P}_1(T)$ for all $T\in \cT_h$, the rest of the proof follows from using \eqref{enrich:1} and \eqref{enrich:2} in \eqref{3.5}.  
\end{proof}
\noindent

\section{Crouzeix-Raviart Analysis of Various State Equations}
This section considers the Crouzeix-Raviart (CR) finite element approximation of the solutions of certain auxiliary state equations. While some of these equations have a typical weak solution in $H^{1}(\Omega)$, others do not. The optimal order $L_{2}$ norm error estimates are derived for both of the cases. To begin with, we define the spaces as follows:
\begin{align*}
W_h&:=\{v_h\in H^1(\O):v_h|_{T}\in \mathbb{P}_1(T) ~\forall ~ T\in \cT_h\},\\
U_h&:=\{u_h\in L_2(\Gamma):u_h|_{e}\in \mathbb{P}_0(e) ~ \forall ~ e\in \cE_h^b\},\\
 V_{1}(\Gamma)&:=\{v_{h}\in C(\Gamma): v_{h}|_{e}\in \mathbb{P}_{1}(e)~ \forall ~ e\in \cE_h^b\}.
\end{align*}
Firstly, we consider the equation which possesses standard weak solution in $H^{1}(\Omega)$. Let $z_{h}\in V_{1}(\Gamma)$ be an arbitrary function. Since, $V_1(\Gamma)\subset H^{1/2}(\Gamma)$, thus $z_h \in H^{\frac{1}{2}}(\Gamma)$. Using $z_h$, we define a function say $\tilde{z}_{h}$ in the following manner:
\begin{align}\label{tildezh}
\begin{cases}
\tilde{z}_h(p)&=0\quad \forall ~p\in \cV_h^i,\\
\tilde{z}_h(p)&=z_h(p)\quad \forall ~p\in \cV_h^b.
\end{cases}
\end{align}
It is easy to see that $\tilde{z}_h\in W_h$. For any $z_h\in V_1(\Gamma)$, define $y_{z_{h}} \in H^1(\O)$ such that:
\begin{align}
y_{z_{h}}&=y_0+\tilde{z}_{h}, ~ y_0\in V,\label{cc:1}\\
a(y_0,v)&=-a(\tilde{z}_{h},v)\quad \forall ~ v\in V.\label{cc:2}
\end{align}
Let $y_{z_{h},h}\in V_h$ be the CR-approximation of $y_{z_{h}}\in H^1(\O)$, i.e. it satisfies the following equation:
\begin{align}
y_{z_{h},h}&={y}_{0,h}+\tilde{z}_{h}, ~ y_{0,h}\in V_h^0,\label{cc:3}\\
a_{pw}({y}_{0,h},v_h)&=-a_{pw}(\tilde{z}_{h},v_h)\quad \forall ~ v_h \in V_h^0.\label{cc:4}
\end{align}
By Lax-milgram lemma, auxiliary problems \eqref{cc:1}-\eqref{cc:2} and \eqref{cc:3}-\eqref{cc:4} are well posed.
%
%
%
The following Theorem yields an optimal order error estimate for $y_{z_{h}}$.:
\begin{theorem}\label{lem2}
For $z_{h}\in V_{1}(\partial\Omega)$ let $y_{z_{h}},\,y_{z_{h},h}$ satisfies \eqref{cc:1}, \eqref{cc:1} and \eqref{cc:3}, \eqref{cc:4} respectively. Then $\|y_{z_{h}}-y_{z_{h},h}\|\lesssim h[\|y_{z_{h}}-y_{z_{h},h}\|_{pw}+h^{1-\epsilon}\|z_{h}\|_{H^{\frac{3}{2}-\epsilon}(\partial\Omega)}]$ for any $\epsilon>0$.
\end{theorem}
\begin{proof}
	Next, we estimate $\|y_{z_h}-y_{z_h,h}\|$. Again, by duality, we find that
	\begin{align}\label{L_2e}
	\|y_{z_h}-y_{z_h,h}\|&=\sup_{g\in L_2(\O), g\neq 0}\frac{(y_{z_h}-y_{z_h,h},g)}{\|g\|}.
	\end{align}
	%
	Using \eqref{cc:1} and \eqref{cc:3}, we have 
	\begin{align}\label{cc:5}
	y_{z_h}-y_{z_h,h}&=y_0-\bar{y}_{0,h}.
	\end{align}
	For any $g\in L_2(\O)$, define $\phi_g\in H^1_0(\O)$ such that
	\begin{align}\label{cc:6}
	a(\phi_g,v)&=(g,v)\quad \forall ~ v\in H^1_0(\O).
	\end{align}
	By elliptic regularity theory, we have $\phi_g\in H^2(\O)$ such that $\|\phi_g\|_{2,\O}\lesssim \|g\|$.
	Let $\phi_{g,h}\in V_h^0$ be the standard non-conforming finite element approximation of $\phi_g$. Thus, it satisfies the following equation:
	\begin{align}\label{cc:7}
	a_{pw}(\phi_{g,h},v_h)&=(g,v_h)\quad \forall ~ v_h \in V_h^0.
	\end{align}
	Since $y_0\in H^1_0(\O)$ and $\bar{y}_{0,h}\in V_h$, it follows from \eqref{cc:2},\eqref{cc:4}, \eqref{cc:6} and \eqref{cc:7} that
	\begin{align}\label{cc:8}
	(g,y_0-\bar{y}_{0,h})&=a(\phi_g,y_0)-a_{pw}(\phi_{g,h},\bar{y}_{0,h})\notag\\
	&=a_{pw}(\phi_g-\phi_{g,h},y_0-\bar{y}_{0,h})+a_{pw}(\phi_g-\phi_{g,h},\bar{y}_{0,h})+a_{pw}(\phi_{g,h},y_0-\bar{y}_{0,h})\notag\\
	&=a_{pw}(\phi_g-\phi_{g,h},y_0-\bar{y}_{0,h})+a_{pw}(\phi_g,\bar{y}_{0,h})-(g,\bar{y}_{0,h})\notag\\
	&\quad +a_{pw}(\phi_{g,h},y_0)+a_{pw}(\tilde{z}_h,\phi_{g,h})\notag\\
	&=a_{pw}(\phi_g-\phi_{g,h},y_0-\bar{y}_{0,h})+a_{pw}(\phi_g,\bar{y}_{0,h}-y_0)-(g,\bar{y}_{0,h}-y_0)\notag\\
	&\quad + a_{pw}(y_0,\phi_{g,h}-\phi_g)+a_{pw}(\tilde{z}_h,\phi_{g,h}-\phi_g)\notag\\
	&=I+II+III+IV+V.
	\end{align}
	We estimate each term on the right hand side of \eqref{cc:8}. First, consider the term $IV+V$. 
	\begin{align*}
	IV+V&=a_{pw}(y_{\tilde{z}_{h}},\phi_{g,h}-\phi_g)\\
	\end{align*}


	Note that since $z_{h}$ is a piecewise linear and globally continuous on each line segment $\Gamma_{j}$ ($1\leq j\leq n$) consisting $\partial\Omega$. Therefore $z_{h}|_{\Gamma_{j}}\in H^{\frac{3}{2}-\epsilon}(\Gamma_{j})$ for any $\epsilon>0$ \cite{nonstandard interpolation}. Since $z_{h}$ satisfies the compatibility conditions that $z_{h}$ is continuous at the vertices of $\Omega$ and as $\Omega$ is convex therefore an application of result from real interpolation spaces \cite{brennerbook}  yields that $y_{z_{h}}\in H^{2-\epsilon}(\Omega)$  \cite{Nonstandard-Regularity}{Theorem 1.8}.
	Consider the term $a_{pw}(y_{\tilde{z}_{h}},\phi_{g,h}-\phi_g)$. A use of integration by parts yields 
	\begin{align}\label{csgd:1}
	a_{pw}(y_{\tilde{z}_{h}},\phi_{g,h}-\phi_g)&=\sum_{T\in \cT_h}\int_{T}\nabla y_{\tilde{z}_{h}}\cdot \nabla (\phi_{g,h}-\phi_g)\,dx\notag\\
	&=-\sum_{T\in \cT_h}\int_T\Delta y_{\tilde{z}_{h}}\,(\phi_{g,h}-\phi_g)\,dx+\sum_{T\in \cT_h}\int_{\partial T}\nabla y_{\tilde{z}_{h}}\cdot n\,(\phi_{g,h}-\phi_g)\,ds\notag\\
	&=-\sum_{T\in \cT_h}\int_T\Delta y_{\tilde{z}_{h}}\,(\phi_{g,h}-\phi_g)\,dx\notag\\
	& \quad+ \sum_{T\in \cT_h}\int_{\partial T}
	\frac{\partial{y}_{\tilde{z}_{h}}}{\partial n}\,(\phi_{g,h}-\phi_g)\,ds,
	\end{align}
	\begin{align}\label{Est-Enr}
	=&-\sum_{T\in \cT_h}\int_T\Delta y_{\tilde{z}_{h}}\,(\phi_{g,h}-\phi_g)\,dx\notag \\
	&+\sum_{T\in \cT_h}\int_{\partial T}\frac{\partial{y}_{\tilde{z}_{h}}}{\partial n} \,((\phi_{g,h}-\phi_g)-I_{CR}((\phi_{g,h}-\phi_g))\,ds+\notag\\&\sum_{T\in \cT_h}\int_{\partial T}\frac{\partial{y}_{\tilde{z}_{h}}}{\partial n} \,(I_{CR}(\phi_{g,h}-\phi_g)-I_cI_{CR}((\phi_{g,h}-\phi_g))ds+\sum_{T\in \cT_h}\int_{\partial T}\frac{\partial{y}_{\tilde{z}_{h}}}{\partial n} \,I_cI_{CR}((\phi_{g,h}-\phi_g))ds\notag \\
	=&-\sum_{T\in \cT_h}\int_T\Delta y_{\tilde{z}_{h}}\,(\phi_{g,h}-\phi_g)\,dx\notag \\
	&+\sum_{T\in \cT_h}\int_{\partial T}(\nabla{y}_{\tilde{z}_{h}}-\Pi_T \nabla{y}_{\tilde{z}_{h}})\cdot n\,((\phi_{g,h}-\phi_g)-I_{CR}((\phi_{g,h}-\phi_g))\,ds+\notag\\&\sum_{T\in \cT_h}\int_{\partial T}(\nabla{y}_{\tilde{z}_{h}}-\Pi_T\nabla{y}_{\tilde{z}_{h}} )\cdot n\,(I_{CR}(\phi_{g,h}-\phi_g)-I_cI_{CR}((\phi_{g,h}-\phi_g))ds+\notag\\
 &\sum_{e\in \cE_h}\int_{e}\frac{\partial{y}_{\tilde{z}_{h}}}{\partial n} \,\jump{I_cI_{CR}((\phi_{g,h}-\phi_g))}ds.
	\end{align}
 	Where $\Pi_T$ is lowest order Raviart-Thomas interpolation operator as in Lemma \ref{RTI}. 
 Note that $\Delta y_{\tilde{z}}=0$ and $\int_{e}\frac{\partial{y}_{\tilde{z}_{h}}}{\partial n} \,\jump{I_cI_{CR}((\phi_{g,h}-\phi_g))}ds=0$. Therefore application of Cauchy Schwarz inequality, discrete trace inequality, Crauzeix Raviart and Raviart Thomas interpolation error estimates in \eqref{Est-Enr} yields
 \begin{align}
     a_{pw}(y_{\tilde{z}_{h}},\phi_{g,h}-\phi_g)&=\lesssim h^{1-\epsilon}|y_{\tilde{z}_{h}}|_{2-\epsilon,\O}\|\phi_{g,h}-\phi_g\|_{pw}.\label{csdg:2}
 \end{align}
	Now, consider the terms $II+III$.
	\begin{align}\label{xoxo:14}
	II+III=&a_{pw}(\phi_g,\bar{y}_{0,h}-y_0)-(g,\bar{y}_{0,h}-y_0)\notag\\
	=&a_{pw}(\phi_g,(\bar{y}_{0,h}-y_0)-I_{CR}(\bar{y}_{0,h}-y_0))+a_{pw}(\phi_g, I_{CR}(\bar{y}_{0,h}-y_0)-I_c(I_{CR}(\bar{y}_{0,h}-y_0)))+\notag\\&a_{pw}(\phi_g,I_c(I_{CR}(\bar{y}_{0,h}-y_0)))-\quad(g,(\bar{y}_{0,h}-y_0)-I_{CR}(\bar{y}_{0,h}-y_0))-\notag\\&(g, I_{CR}(\bar{y}_{0,h}-y_0)-I_c(I_{CR}(\bar{y}_{0,h}-y_0)))-(g,I_c(I_{CR}(\bar{y}_{0,h}-y_0))).
	\end{align}
	From \eqref{xoxo:14} consider the term $a_{pw}(\phi_g,(\bar{y}_{0,h}-y_0)-I_{CR}(\bar{y}_{0,h}-y_0))$.
	\begin{align}\label{II-III-estimate}
	&a_{pw}(\phi_g,(\bar{y}_{0,h}-y_0)-I_{CR}(\bar{y}_{0,h}-y_0))\notag\\
	=&-\sum_{T\in \cT_h} \int_{T}\Delta\phi_{g}((\bar{y}_{0,h}-y_0)-I_{CR}(\bar{y}_{0,h}-y_0))\,dx+\sum_{T\in \cT_h}\int_{\partial T}\dfrac{\partial\phi_{g}}{\partial n}((\bar{y}_{0,h}-y_0)-I_{CR}(\bar{y}_{0,h}-y_0))\,ds\notag\\
	=&\sum_{T\in\cT_h}\int_{T}{g}((\bar{y}_{0,h}-y_0)-I_{CR}(\bar{y}_{0,h}-y_0))\,dx+\sum_{T\in \cT_h}\int_{\partial T}(\nabla\phi_g-(\Pi_T\nabla\phi_g))\cdot n_{e}\notag\\&((\bar{y}_{0,h}-y_0)-I_{CR}(\bar{y}_{0,h}-y_0))\,ds\notag\\\leq&\|g\|\|(\bar{y}_{0,h}-y_0)-I_{CR}(\bar{y}_{0,h}-y_0)\|+\sum_{e\in\mathcal{E}_h}\|(\nabla\phi_g-(\Pi_T\nabla\phi_g))\cdot n_{e}\|_{\partial T_e}\|(\bar{y}_{0,h}-y_0)-I_{CR}(\bar{y}_{0,h}-y_0)\|_{\partial T_e}
	\end{align}
    An application of the discrete trace-inequality and the CR interpolation error estimation in \eqref{II-III-estimate} yields,
    \begin{align}\label{II-III-estimate11}
        &a_{pw}(\phi_g,(\bar{y}_{0,h}-y_0)-I_{CR}(\bar{y}_{0,h}-y_0))\lesssim h\|g\|\|\bar{y}_{0,h}-y_0\|_{pw}.
    \end{align}
    Arguments in the similar line yields 
    \begin{align}\label{II-III-estimate12}
        a_{pw}(\phi_g, I_{CR}(\bar{y}_{0,h}-y_0)-I_c(I_{CR}(\bar{y}_{0,h}-y_0)))\lesssim h\|g\|\|\bar{y}_{0,h}-y_0\|_{pw}.
    \end{align}
    Combining \eqref{cc:6}, \eqref{xoxo:14}, \eqref{II-III-estimate}, \eqref{II-III-estimate11}, \eqref{II-III-estimate12} and standard $L_2$ norm error-estimates for $I_{CR}$ interpolation operator yields,
    \begin{align}\label{xoxo:15}
	II+III&\lesssim \|g\|\|\bar{y}_{0,h}-y_0\|_{pw}.
	\end{align}

	Next we aim to consider $\|y_{0,h}-y_0\|_{pw}$.\\
	From \eqref{cc:1}, \eqref{cc:2}, \eqref{cc:3}, Theorem \ref{Orthog1} and \eqref{cc:4} we obtain
	\begin{align}
	\|y_{0,h}-I_{CR}y_0\|^{2}_{pw}&\leq a_{pw}(I_{CR}y_0-y_{0,h}, I_{CR}y_0-y_{0,h})\notag\\
	&= a_{pw}(I_{CR}y_0, I_{CR}y_0-y_{0,h})- a_{pw}(y_{0,h}, I_{CR}y_0-y_{0,h})\notag\\
	&=a_{pw}(I_{CR}y_0-y_0, I_{CR}y_0-y_{0,h})+a_{pw}(y_0, I_{CR}y_0-y_{0,h})+\notag\\ &\quad \,\,a_{pw}(\tilde{P_{1}}(\bar{u}_h), I_{CR}y_0-y_{0,h}-I_c(I_{CR}y_0-y_{0,h}))+a_{pw}(\tilde{P_{1}}(\bar{u}_h),I_c(I_{CR}y_0-y_{0,h}))\notag\\
	&=a_{pw}(I_{CR}y_0-y_0, I_{CR}y_0-y_{0,h})+a_{pw}(y_0, I_{CR}y_0-y_{0,h})+a_{pw}(\tilde{P_{1}}(\bar{u}_h),I_c(I_{CR}y_0-y_{0,h}))\notag\\
	&=a_{pw}(I_{CR}y_0-y_0, I_{CR}y_0-y_{0,h})+a_{pw}(y_0, I_{CR}y_0-y_{0,h}-I_{c}(I_{CR}y_0-y_{0,h}))+\notag\\&\quad \,\,a_{pw}(y_0, I_{c}(I_{CR}y_0-y_{0,h}))+a_{pw}(\tilde{P_{1}}(\bar{u}_h),I_c(I_{CR}y_0-y_{0,h}))\notag\\&=a_{pw}(I_{CR}y_0-y_0, I_{CR}y_0-y_{0,h})+a_{pw}(y_0, I_{CR}y_0-y_{0,h}-I_{c}(I_{CR}y_0-y_{0,h}))\notag\\
	&=a_{pw}(I_{CR}y_0-y_0, I_{CR}y_0-y_{0,h})+a_{pw}(y_0-I_{CR}y_0, I_{CR}y_0-y_{0,h}-I_{c}(I_{CR}y_0-y_{0,h}))
	\label{Est-IV1}
	\end{align}
	Then applying the similar arguments as in \eqref{csgd:1}, \eqref{Est-Enr} and Young's inequality in the right hand side of \eqref{Est-IV1} we obtain 
	\begin{align}\label{BestErEs}
	\|y_{0,h}-I_{CR}y_0\|_{pw}\lesssim \|y_0-I_{CR}y_0\|_{pw}.
	\end{align}
 Finally combining \eqref{L_2e}, \eqref{cc:5}, \eqref{cc:8}, \eqref{csdg:2}, \eqref{xoxo:15} and \eqref{BestErEs} the result is obtained.
\end{proof}
{\bf Corollary}:   For $z_{h}\in V_{1}(\partial\Omega)$ let $y_{z_{h}},\,y_{z_{h},h}$ satisfies \eqref{cc:1}-\eqref{cc:2} and \eqref{cc:3}-\eqref{cc:4} respectively. Then $\|y_{z_{h}}-y_{z_{h},h}\|\lesssim h[\|y_{z_{h}}-y_{z_{h},h}\|_{pw}+\|z_{h}\|_{H^{\frac{1}{2}}(\partial\Omega)}]$.

We note that since we are considering triangulations $\mathcal{T}_h$ such that it induces an odd number of edges on $\partial\Omega$ and $dim ~V_1(\partial\Omega)=dim~ U_h$ therefore $P_0:V_1(\partial\Omega)\rightarrow U_h$ is bijective. Define $\tilde{P}_1:U_h\rightarrow V_1(\partial\Omega)$ by
\begin{align}\label{asso opr}
    \tilde{P}_1=P_0^{-1}.
\end{align}
Note that the matrix representation of $P_0: V_1(\partial\Omega)\rightarrow U_h$ is 
$$
\begin{bmatrix}
 1 & 0 & 0 & \cdot & \cdot & 0 \\
 1 & 1 & 0 & \cdot & \cdot & 0 \\
 0 & 1 & 1 & \cdot & \cdot & 0 \\
 \cdot & \cdot & \cdot & \cdot & \cdot & 0\\
 \cdot & \cdot & \cdot & 1 & 1 & 0\\
 0 & 0 & \cdot & 0 & 1 & 1
\end{bmatrix}
$$
Therefore if $\|\tilde{P}_1\|_{\mathcal{L}(U_h,V_1(\partial\Omega))}$ denotes the operator norm of $\tilde{P}_1$ subordinate to $L_2$ norm then it is uniformly bounded with respect to the mesh refinement parameter $h>0$ \cite{Ipsn}.
$i.\;e.$
\begin{align}\label{unifbdd1}
    \|\tilde{P}_1\|_{\mathcal{L}(U_h,V_1(\partial\Omega))}\leq C,
\end{align}
where $C>0$ is independent of $h>0$. This operator plays a central role in associating the discrete control variable with the discrete state variable, which is discussed in the next section.


\begin{theorem}\label{lem:1}
	For a piecewise constant function ${u}_h$ defined on $\partial\Omega$, the following holds:
	\begin{align*}
	\|y_{{u}_h}-{y}_{\tilde{P}_{1}(u_{h})}\|\lesssim {h}\|\tilde{P}_{1}{u}_h\|_{H^{\frac{1}{2}}(\partial\Omega)}..
	\end{align*}
	    \textcolor{red}{	
		\begin{align*}
		\|y_{{u}_h}-{y}_{P_{1}(u_{h})}\|\lesssim h.
		\end{align*}}

\end{theorem}
\begin{proof}
	Since ${u}_h\in L_2(\Gamma)$, $y_{{u}_h}\in L_2(\O)$ is the solution of the following ultra weak formulation:
	\begin{align}\label{c:2}
	(y_{u_h},\Delta \phi)&=-\langle {u}_h,\partial \phi/\partial n\rangle \quad \forall~ \phi \in X.
	\end{align}
	Note that $\tilde{P}_{1}({u}_h)\in H^{1/2}(\Gamma)$ and $y_{\tilde{P}_{1}(\bar{u}_h)}\in L_2(\O)$ also satisfies
	\begin{align}\label{c:1}
	(y_{\tilde{P}_{1}({u}_h)},\Delta \phi)&=-\langle \tilde{P}_{1}({u}_h),\partial \phi/\partial n\rangle \quad \forall ~ \phi \in X. 
	\end{align}
	By duality, we have
	\begin{align}\label{c:5}
	\|y_{{u}_h}-y_{\tilde{P}_{1}({u}_h)}\|&=\sup_{g\in L_2(\O),g\neq 0}\frac{(g,y_{{u}_h}-y_{\tilde{P}_{1}({u}_h)})}{\|g\|}.
	\end{align}

	For any $\phi \in X$, it follows from \eqref{c:2}, \eqref{c:1} that 
	\begin{align}
	(y_{\tilde{P}_{1}({u}_h)}-y_{{u}_h},\Delta \phi)&=\langle {u}_h- \tilde{P}_{1}({u}_h),\partial \phi/\partial n\rangle \notag\\
	&= \langle P_0(\tilde{P}_{1}{u}_h)- \tilde{P}_{1}({u}_h),\partial \phi/\partial n-P_{0}(\partial \phi/\partial n)\rangle.\label{c:3}
	\end{align} 
	We know that for any $g\in L_2(\O)$, there exists $\phi \in X$ such that $-\Delta \phi=g ~ \text{in}~ \O$ and since $\Omega$ is convex from elliptic regularity theory $\|\phi\|_{H^{2}(\Omega)}\lesssim \|g\|$. Thus, it follows from \eqref{c:3} and Cauchy-Schwarz inequality that, 
	\begin{align}
	(y_{{u}_h}-y_{\tilde{P}_{1}({u}_h)},g)&=\langle {u}_h- \tilde{P}_{1}({u}_h),\partial \phi/\partial n-P_{0}(\partial \phi/\partial n)\rangle\notag\\
	&=\langle P_0(\tilde{P}_{1}{u}_h)- \tilde{P}_{1}({u}_h),\partial \phi/\partial n-P_{0}(\partial \phi/\partial n)\rangle\notag\\
	&\lesssim \|P_0(\tilde{P}_{1}{u}_h)-\tilde{P}_{1}{u}_h\|_{0,\Gamma}\|\partial \phi/\partial n-P_{0}(\partial \phi/\partial n)\|_{0,\Gamma}.\label{c:4}
	\end{align}
	Then combining \eqref{c:5}, \eqref{c:4} and the elliptic regularity estimate for $\phi$ we obtain
	\begin{align}\label{c:8}
	\|y_{{u}_h}-y_{P_{1}({u}_h)}\|&\lesssim {h}\|\tilde{P}_{1}{u}_h\|_{H^{\frac{1}{2}}(\partial\Omega)}.
	\end{align}

	\end{proof}
\begin{theorem}\label{supconv8}
	For a piecewise constant function $u_{h}$ defined on $\partial\Omega$ it's extension $y_{u_{h}}$ and it's CR approximation $y_{u_{h},h}$ satisfies the following estimate:
	\begin{align*}
	\|y_{u_{h}}-y_{u_{h},h}\|\lesssim h.
	\end{align*}
\end{theorem}	
\begin{proof}
		 Applying Theorems \ref{lem2}, \ref{lem:1} the proof follows .
\end{proof}	
\begin{theorem}\label{aux1}
	The following estimate holds:
	\begin{align*}
	\|y_{P_0\bar{u}}-y_{\tilde{P}_1P_0\bar{u}}\|\lesssim h\|\bar{u}\|_{H^{\frac{1}{2}}(\partial\Omega)}.
	\end{align*}
\end{theorem}
\begin{proof}
	The main ingredients of the proof consists of orthogonality properties of the $L_2$ projection operator and bijection of $\tilde{P}_{1}$. Note that since we consider the triangulation such that it induces an odd numbers of edges on $\partial\Omega$ therefore given ${P}_{0}\bar{u}\in U_h$ there exists an unique $v_1\in V_1(\partial\Omega)$ such that
	\begin{align}\label{bijection}
	{P}_{0}\bar{u}=P_0v_1.
	\end{align}
    Note that 
	\begin{align}\label{l2defn}
	\|y_{P_0\bar{u}}-y_{\tilde{P}_{1}P_0\bar{u}}\|=\sup_{g\in L_2(\Omega), g\neq 0}\dfrac{(y_{P_0\bar{u}}-y_{\tilde{P}_{1}P_0\bar{u}},g)}{\|g\|}.
	\end{align}
	Then consider the auxiliary problem:
	\begin{align}\label{ax1}
	-\Delta\phi&=g~\text{in}~\Omega,\\
	\phi&=0~\text{on}~\partial\Omega.\notag
	\end{align}
	Considering $(y_{P_0\bar{u}}-y_{\tilde{P}_{1}P_0\bar{u}},g)$, \eqref{ax1}, $P_0^2=P_0$ along with integration by parts implies
	\begin{align}\label{l2est2}
	(y_{\tilde{P}_{1}P_0\bar{u}}-y_{P_0\bar{u}},g)&=(\tilde{P}_{1}P_0\bar{u}-P_0\bar{u},\dfrac{\partial\phi}{\partial n})\notag\\&=(\tilde{P}_{1}P_0v_1-P_0v_1,\dfrac{\partial\phi}{\partial n})\notag\\&=(v_1-P_0v_1,\dfrac{\partial\phi}{\partial n})\notag\\&=(v_1-P_0v_1,\dfrac{\partial\phi}{\partial n}-P_0\dfrac{\partial\phi}{\partial n}).
	\end{align}
	Therefore combination of \eqref{l2defn}, \eqref{bijection} and \eqref{l2est2} implies that 
	\begin{align}\label{ax2}
	\|y_{P_0\bar{u}}-y_{\tilde{P}_{1}P_0\bar{u}}\|\lesssim h\|\tilde{P}_{1}P_0\bar{u}\|_{H^{\frac{1}{2}}(\partial\Omega)}.
	\end{align}
	Next we show that $\|\tilde{P}_{1}P_0\bar{u}\|_{H^{\frac{1}{2}}(\partial\Omega)}\lesssim \|\bar{u}\|_{H^{\frac{1}{2}}(\partial\Omega)}$. Consider
	\begin{align}\label{l2est3}
	\|\tilde{P}_{1}P_0\bar{u}\|_{H^{\frac{1}{2}}(\partial\Omega)}&\leq\|\tilde{P}_{1}P_0(\bar{u}-\pi_1\bar{u})\|_{H^{\frac{1}{2}}(\partial\Omega)}+\|\tilde{P}_{1}P_0(\pi_1\bar{u})\|_{H^{\frac{1}{2}}(\partial\Omega)}\notag\\&\leq Ch^{-\frac{1}{2}}\|\tilde{P}_{1}P_0(\bar{u}-\pi_1\bar{u})\|+\|\pi_1\bar{u}\|_{H^{\frac{1}{2}}(\partial\Omega)}\notag\\&\leq Ch^{-\frac{1}{2}}\|P_0(\bar{u}-\pi_1\bar{u})\|+\|\pi_1\bar{u}\|_{H^{\frac{1}{2}}(\partial\Omega)}\notag\\&\leq Ch^{-\frac{1}{2}}\|\bar{u}-\pi_1\bar{u}\|+\|\bar{u}\|_{H^{\frac{1}{2}}(\partial\Omega)}\notag\\&\leq C\|\bar{u}\|_{H^{\frac{1}{2}}(\partial\Omega)}.
	\end{align}
	Therefore combining \eqref{l2defn}, \eqref{l2est2}, \eqref{ax2} and \eqref{l2est3} the result is obtained.
\end{proof}

\begin{theorem}\label{supconv2}
	Let $z\in H^{\frac{1}{2}}(\partial\Omega)$, then the following holds:
	\begin{align*}
	\|y_{z}-y_{P_0(z),h}\|\lesssim h\|z\|_{H^{1/2}(\Gamma)}.
	\end{align*}
\end{theorem}
\begin{proof}
	By triangle inequality, we have
	\begin{align}\label{hh:3}
	\|y_{z}-y_{P_0(z),h}\|\leq \|y_{z}-y_{\tilde{P}_1(P_0(z))}\|+\|y_{\tilde{P}_1(P_0(z))}-y_{P_0(z),h}\|.
	\end{align}
Consider $\|y_z-y_{\tilde{P}_1(P_0z)}\|$.
\begin{align}\label{4.28}
   \|y_z-y_{\tilde{P}_1(P_0z)}\|\leq \|y_z-y_{P_0z}\|+\|y_{P_0z}-y_{\tilde{P}_1(P_0z)}\|. 
\end{align}
Consider $\|y_{P_0z}-y_{\tilde{P}_1(P_0z)}\|$.
We write the $L_2$ norm in duality form 
\begin{align}\label{4.29}
   \|y_{P_0z}-y_{\tilde{P}_1(P_0z)}\|=\sup_{g\in L_2(\Omega), g\neq 0}\dfrac{(y_{P_0z}-y_{\tilde{P}_1(P_0z)},g)}{\|g\|} 
\end{align}
Consider the auxiliary problem 
\begin{align}\label{4.30}
    -\Delta\phi&=g~\text{in}~\Omega,\\
    \phi&=0~\text{on}~\partial\Omega.\notag
\end{align}
Consider $(y_{P_0z}-y_{\tilde{P}_1(P_0z)},g)$. Combining \eqref{4.29}, \eqref{4.30} and integration by parts we obtain 
\begin{align}\label{4.31}
   (y_{P_0z}-y_{\tilde{P}_1(P_0z)},g)&=(y_{\tilde{P}_1(P_0z)}-y_{P_0z},\Delta\phi)\notag\\
   &=(P_0z-\tilde{P}_1(P_0z),\dfrac{\partial\phi}{\partial n})\notag\\
&=(P_0\tilde{P}_1(P_0z))-\tilde{P}_1(P_0z),\dfrac{\partial\phi}{\partial n})\notag\\
&=(P_0(\tilde{P}_1P_0z)-\tilde{P}_1P_0z,\dfrac{\partial\phi}{\partial n}-P_0\dfrac{\partial\phi}{\partial n})\notag\\
&\lesssim h\|\tilde{P}_1P_0z\|_{H^{\frac{1}{2}}(\partial\Omega)}\|g\|
\end{align}
Therefore combining \eqref{4.29}, \eqref{4.31} we obtain
\begin{align}\label{4.32}
  \|y_{P_0z}-y_{\tilde{P}_1(P_0z)}\|\lesssim h\|\tilde{P}_1P_0z\|_{H^{\frac{1}{2}}(\partial\Omega)}. 
\end{align}
Next consider $\|y_z-y_{P_0z}\|$. Applying duality arguments as above we obtain
\begin{align}\label{4.33}
    \|y_z-y_{P_0z}\|\lesssim h\|z\|_{H^{\frac{1}{2}}(\partial\Omega)}.
\end{align}
Applying Corollary 1 we obtain
\begin{align}\label{4.34}
    \|y_{\tilde{P}_1(P_0(z))}-y_{P_0(z),h}\|\lesssim h\|\tilde{P}_1P_0(z)\|_{H^{\frac{1}{2}}(\partial\Omega)}.
\end{align}
Combining \eqref{l2est3}, \eqref{hh:3}, \eqref{4.28}, \eqref{4.31}, \eqref{4.32}, \eqref{4.33}, \eqref{4.34}  we obtain the result.
\end{proof}

\section{Discrete Control Problem and Error Estimate}
In this section the fully discrete control problem is introduced. The control space is discretized by piecewise constant functions. Define the discrete control problem by:
\begin{align}
J_{h}(y_{h},u_{h})=\dfrac{1}{2}\|y_{h}-y_{d}\|^{2}+\dfrac{\alpha}{2}\|\tilde{P}_1(u_h)\|^{2},
\end{align}
subject to,
\begin{align}
&{y}_h={y}_{f,h}+{y}_{{u}_h,h}, \quad {y}_{f,h} \in V_h^0,~ {y}_{{u}_h,h}={y}_{0,h}+\tilde{P_{1}}({u}_h) \in V_h,\\
&a_{pw}({y}_{f,h}, v_h)=(f,v_h) \quad \forall ~ v_h \in V_h^0,\\
&a_{pw}({y}_{0,h}, v_h)=-a_{pw}(\widetilde{\tilde{P_{1}}({u}_h)}, v_h) \quad \forall ~ v_h \in V_h^0.
\end{align}
The existence and uniqueness of the solution of the problem follows from finite dimensional and strict convexity arguments.  The optimal variables $(\bar{u}_{h}, \bar{y}_{h})\in U_{h, ad}\times V_{h}$ satisfies the following system

\begin{theorem}
There exists a unique solution $(\bar{y}_h, \bar{u}_h)\in V_h\times U_{h,ad}$ of the discrete optimal control problem. Furthermore, it satisfy the following system of equations:
\begin{align}
&\bar{y}_h=\bar{y}_{f,h}+\bar{y}_{\bar{u}_h,h}, \quad \bar{y}_{f,h} \in V_h^0,~ \bar{y}_{\bar{u}_h,h}=\bar{y}_{0,h}+\tilde{P_{1}}(\bar{u}_h) \in V_h,\label{diskkt:1}\\
&a_{pw}(\bar{y}_{f,h}, v_h)=(f,v_h) \quad \forall ~ v_h \in V_h^0,\label{diskkt:2}\\
&a_{pw}(\bar{y}_{0,h}, v_h)=-a_{pw}(\widetilde{\tilde{P_{1}}(\bar{u}_h)}, v_h) \quad \forall ~ v_h \in V_h^0,\label{diskkt:3}\\
&(\bar{y}_h-y_d,y_{p_h,h}-\bar{y}_{\bar{u}_h,h})+\alpha\langle \tilde{P}_{1}(\bar{u}_h),\tilde{P}_{1}(p_h)-\tilde{P}_1(\bar{u}_h)\rangle\geq 0\quad \forall ~ p_h \in U_{h,ad}.\label{diskkt:4}
\end{align}
\end{theorem}
Next we prove the following auxiliary estimate which is important to obtain the optimal order error estimate for the optimal control.
\begin{theorem}\label{sup-conv}
	For the discrete optimal control $\bar{u}_h$ the following error estimate holds:
	\begin{align*}
	\|y_{\bar{u}_h}-y_{\bar{u}_h,h}\|\lesssim h
	\end{align*}
\end{theorem}
\begin{proof}
From Theorem \ref{lem:1} we obtain
	\begin{align}\label{c::8}
	\|y_{\bar{u}_h}-y_{\tilde{P}_{1}(\bar{u}_h)}\|&\lesssim h\|\tilde{P}_{1}(\bar{u}_h)\|_{H^{\frac{1}{2}}(\partial\Omega)}.
	\end{align}
	Next it remains to find an uniform bound for $\|\tilde{P}_{1}(\bar{u}_h)\|_{H^{\frac{1}{2}}(\partial\Omega)}$.
	Let $\xi_{h}\in V_h$ satisfies 
	\begin{align}
	a_{pw}(\xi_h,v_h)=(\bar{y}_h-y_d,v_h),\forall v_h\in V_h
	\end{align}
	Then note that \eqref{diskkt:4} is equivalent to 
	\begin{align}\label{dvie}
	(\bar{\theta}_h+\alpha \tilde{P}_1\bar{u}_h,\tilde{P}_1p_h-\tilde{P}_1\bar{u}_h)\geq  0,\forall p_h\in U_{ad},
	\end{align}
	where for any $r_h\in V_1(\partial\Omega)$, $\bar{\theta}_h\in V_1(\partial\Omega)$ is defined by
	\begin{align}\label{adjest1}
	(\bar{\theta}_h,r_h)=a_{pw}(\bar{\xi}_h,\tilde{r}_h)-(\bar{y}_h-y_d,\tilde{r}_h),
	\end{align}
	where $\tilde{r}_h\in V_h$ is a piecewise linear, globally continuous extension of $r_h$ in $\Omega$.
	Next we prove 
	\begin{align}\label{adjest2}
	\|\dfrac{\partial\theta_{y_{\bar{u}_h}}}{\partial n}-\bar{\theta}_h\|\lesssim \sqrt{h}.
	\end{align}
	Where $\theta_{y_{\bar{u}_h}}\in H^1_0(\Omega)$ satisfies
	\begin{align}\label{adjest3}
	-\Delta\theta_{y_{\bar{u}_h}}&=y_{\bar{u}_h}-y_d \;\text{in}\;\Omega,\\
	\theta_{y_{\bar{u}_h}}&=0\;\text{on}\;\partial\Omega.\notag
	\end{align}
	Note that 
	\begin{align}\label{adjest4}
	\|\dfrac{\partial\theta_{y_{\bar{u}_h}}}{\partial n}-\bar{\theta}_h\|^2\lesssim\|\dfrac{\partial\theta_{y_{\bar{u}_h}}}{\partial n}-\tilde{P}_1P_0(\dfrac{\partial\theta_{y_{\bar{u}_h}}}{\partial n})\|^2+\|\tilde{P}_1P_0(\dfrac{\partial\theta_{y_{\bar{u}_h}}}{\partial n})-\bar{\theta}_h\|^2
	\end{align}
 Also  
 \begin{align}\label{adjest4.6}
     \|\tilde{P}_1P_0(\dfrac{\partial\theta_{y_{\bar{u}_h}}}{\partial n})-\bar{\theta}_h\|^2\lesssim      \|\tilde{P}_1P_0(\dfrac{\partial\theta_{y_{\bar{u}_h}}}{\partial n})-\dfrac{\partial\theta_{y_{\bar{u}_h}}}{\partial n}\|^2+\|\dfrac{\partial\theta_{y_{\bar{u}_h}}}{\partial n}-P_1\dfrac{\partial\theta_{y_{\bar{u}_h}}}{\partial n}\|^2+\|P_1\dfrac{\partial\theta_{y_{\bar{u}_h}}}{\partial n}-\bar{\theta}_h\|^2.
 \end{align}
	Consider,
	\begin{align}\label{adjest5}
	\|P_1(\dfrac{\partial\theta_{y_{\bar{u}_h}}}{\partial n})-\bar{\theta}_h\|^2=(\dfrac{\partial\theta_{y_{\bar{u}_h}}}{\partial n}-\bar{\theta}_h,P_1(\dfrac{\partial\theta_{y_{\bar{u}_h}}}{\partial n})-\bar{\theta}_h)
	\end{align}
	To estimate $\|\dfrac{\partial\theta_{y_{\bar{u}_h}}}{\partial n}-\bar{\theta}_h\|$ consider the following auxiliary problem:
	Find $z_h\in \tilde{V}_h$ such that
	\begin{align}\label{adjest6}
	 a_{pw}(z_h,v_h)&=0, \forall v_h\in V_h\\
	 z_h&=P_1(\dfrac{\partial\theta_{y_{\bar{u}_h}}}{\partial n})-\bar{\theta}_h\;\text{on}\;\partial\Omega.\notag
	\end{align} 
	Then from \eqref{adjest1}, \eqref{adjest6} and \eqref{adjest3} we obtain,
	\begin{align}\label{adjest7}
	\int_{\partial\Omega}\left(\dfrac{\partial\theta_{y_{\bar{u}_h}}}{\partial n}-\bar{\theta}_h\right)\left(P_1\left(\dfrac{\partial\theta_{y_{\bar{u}_h}}}{\partial n}\right)-\bar{\theta}_h\right)&=a_{pw}(\theta_{y_{\bar{u}_h}}-\xi_h,z_h)\notag\\
 &=a_{pw}(\theta_{y_{\bar{u}_h}}-I_{CR}\theta_{y_{\bar{u}_h}},z_h)\notag\\
	&\leq\|\theta_{y_{\bar{u}_h}}-I_{CR}\theta_{y_{\bar{u}_h}}\|_{pw}\|z_h\|_{pw}\notag\\&\lesssim h\|y_{\bar{u}_h}-y_d\|]\|z_h\|_{pw}\notag\\&
 \lesssim h\|\bar{u}_h\|_{H^{-\frac{1}{2}}(\partial\Omega)}\|P_1(\dfrac{\partial\theta_{y_{\bar{u}_h}}}{\partial n})-\bar{\theta}_h\|_{H^{\frac{1}{2}}(\partial\Omega)}\notag\\&
 \lesssim \sqrt{h}\|\bar{u}_h\|\|P_1(\dfrac{\partial\theta_{y_{\bar{u}_h}}}{\partial n})-\bar{\theta}_h\|
	\end{align}
 Combining \eqref{adjest5} and \eqref{adjest7} and applying the fact that $\|\bar{u}_h\|$ is uniformly bounded with respect to $h>0$ we obtain 
\begin{align}\label{adjest8}
    \|\dfrac{\partial\theta_{y_{\bar{u}_h}}}{\partial n}-\bar{\theta}_h\|\lesssim \sqrt{h}.
\end{align}
Hence application of triangle inequality, interpolation error estimates and inverse inequality implies that
\begin{align}\label{unifbdd2}
    \|\bar{\theta}_h\|_{H^{\frac{1}{2}}(\partial\Omega)}\leq C.
\end{align}
Note that \eqref{dvie} implies
\begin{align}
    \tilde{P}_1\bar{u}_h=P_{\tilde{P}_1(U_{h,ad})}(\dfrac{1}{\alpha}\bar{\theta}_h).
\end{align}
Also from \eqref{unifbdd1} it is evident that $\tilde{P}_1(U_{h,ad})$ is an uniformly bounded convex subset of $V_1(
\partial\Omega)$. Therefore by  the stability of $P_{\tilde{P}_1(U_{h,ad})}$ with respect to $H^{\frac{1}{2}}(\partial\Omega)$ norm \cite{casasmateosraymond:2009} we obtain that 
\begin{align}\label{unifbdd3}
    \|\tilde{P}_1\bar{u}_h\|_{H^{\frac{1}{2}}(\partial\Omega)}\leq C.
\end{align}
Combining \eqref{c::8} and \eqref{unifbdd3} the result follows. 
\end{proof}

\begin{theorem}
Suppose $\bar{u}\in U_{ad}$, $\bar{u}_h\in U_{h,ad}$ be the optimal control and discrete optimal control respectively. The following holds:
\begin{align*}
\|\bar{u}-\bar{u}_h\|_{L_2(\Gamma)}\lesssim h^{1/2}.
\end{align*}
\end{theorem}
\begin{proof}
We begin by choosing $p_{h}=P_0\bar{u}$ in \eqref{diskkt:4} to obtain,

\begin{align}\label{cntest1}
    \alpha\|\tilde{P}_1\bar{u}_h-\tilde{P}_1P_0\bar{u}\|^2\leq& \alpha(\tilde{P}_1P_0\bar{u},\tilde{P}_1P_0\bar{u}-\tilde{P}_1\bar{u}_h)+(\bar{y}_h-y_d, y_{P_0\bar{u},h}-y_{\bar{u}_h,h})\notag\\=&\alpha(\tilde{P}_1P_0\bar{u}-P_0\bar{u},\tilde{P}_1P_0\bar{u}-\tilde{P}_1\bar{u}_h)+(P_0\bar{u},\tilde{P}_1P_0\bar{u}-\tilde{P}_1\bar{u}_h)+(\bar{y}_h-y_d, y_{P_0\bar{u},h}-y_{\bar{u}_h,h})\notag\\=&\alpha(\tilde{P}_1P_0\bar{u}-P_0(\tilde{P}_1P_0\bar{u}),\tilde{P}_1P_0\bar{u}-\tilde{P}_1\bar{u}_h)+\alpha(P_0\bar{u},\tilde{P}_1P_0\bar{u}-\tilde{P}_1\bar{u}_h)+\notag\\&(\bar{y}_h-\bar{y}, y_{P_0\bar{u},h}-y_{\bar{u}_h,h})+(\bar{y}-y_d, y_{P_0\bar{u},h}-y_{\bar{u}_h,h})
    \end{align}

Application of Young's inequality to \eqref{cntest1} and the fact that $P_0\tilde{P}_1=I_{U_h}$ yields,
\begin{align}\label{cntest2}
    (\alpha-\epsilon)\|\tilde{P}_1\bar{u}_h-\tilde{P}_1P_0\bar{u}\|^2\leq &\alpha(P_0\bar{u},\tilde{P}_1P_0\bar{u}-\tilde{P}_1\bar{u}_h)+(\bar{y}_h-\bar{y}, y_{P_0\bar{u},h}-y_{\bar{u}_h,h})+\notag\\&(\bar{y}-y_d, y_{P_0\bar{u},h}-y_{\bar{u}_h,h})+\|\tilde{P}_1P_0\bar{u}-P_0(\tilde{P}_1P_0\bar{u})\|^2\notag\\
    =&\alpha(P_0\bar{u},(\tilde{P}_1P_0\bar{u}-\tilde{P}_1\bar{u}_h)-P_0(\tilde{P}_1P_0\bar{u}-\tilde{P}_1\bar{u}_h))+(\bar{y}_h-\bar{y}, y_{P_0\bar{u},h}-y_{\bar{u}_h,h})+\notag\\&(\bar{y}-y_d, y_{P_0\bar{u},h}-y_{\bar{u}_h,h})+\|\tilde{P}_1P_0\bar{u}-P_0(\tilde{P}_1P_0\bar{u})\|^2+\alpha (P_0\bar{u}, P_0\bar{u}-\bar{u}_h)\notag\\
    =&(\bar{y}_h-\bar{y}, y_{P_0\bar{u},h}-y_{\bar{u}_h,h})+(\bar{y}-y_d, y_{P_0\bar{u},h}-y_{\bar{u}_h,h})+\notag\\&\frac{\alpha}{\epsilon}\|\tilde{P}_1P_0\bar{u}-P_0(\tilde{P}_1P_0\bar{u})\|^2+\alpha (P_0\bar{u}, P_0\bar{u}-\bar{u}_h)
    \end{align} 
From \eqref{cntest2} consider $(\bar{y}-y_d, y_{P_0\bar{u},h}-y_{\bar{u}_h,h})+\alpha (P_0\bar{u}, P_0\bar{u}-\bar{u}_h)$.
\begin{align}\label{cntest3}
    (\bar{y}-y_d, y_{P_0\bar{u},h}-y_{\bar{u}_h,h})+\alpha (P_0\bar{u}, P_0\bar{u}-\bar{u}_h)=&(\bar{y}-y_d, y_{P_0\bar{u},h}-y_{\bar{u}_h,h})+\notag\\&\alpha (P_0\bar{u}-\bar{u}, P_0\bar{u}-\bar{u}_h)+\alpha(\bar{u},P_0\bar{u}-\bar{u}_h)\notag\\=&(\bar{y}-y_d, y_{P_0\bar{u},h}-y_{\bar{u}_h,h})+\alpha(\bar{u},P_0\bar{u}-\bar{u}_h)\notag\\=&(\bar{y}-y_d, y_{P_0\bar{u},h}-y_{\bar{u}_h,h})+\alpha(\bar{u},P_0\bar{u}-\bar{u})+\alpha(\bar{u},\bar{u}-\bar{u}_h)\notag\\
    =&-\alpha\|\bar{u}-P_0\bar{u}\|^2+\alpha(P_0\bar{u},P_0\bar{u}-\bar{u})+\notag\\&(\bar{y}-y_d, y_{P_0\bar{u},h}-y_{\bar{u}_h,h})+\alpha(\bar{u},\bar{u}-\bar{u}_h)\notag\\
    \leq&(\bar{y}-y_d, y_{P_0\bar{u},h}-y_{\bar{u}_h,h})+\alpha(\bar{u},\bar{u}-\bar{u}_h)\notag\\
    =&\alpha(\bar{u},\bar{u}-\bar{u}_h)+(\bar{y}-y_d, y_{P_0\bar{u},h}-y_{P_0\bar{u}}+\notag\\&y_{P_0\bar{u}}-y_{\bar{u}}+y_{\bar{u}}-y_{\bar{u}_h}+y_{\bar{u}_h}-y_{\bar{u}_h,h}).
\end{align}
Combining \eqref{ctskkt:4} and \eqref{cntest3} we obtain
\begin{align}\label{cntest4}
        (\bar{y}-y_d, y_{P_0\bar{u},h}-y_{\bar{u}_h,h})+\alpha (P_0\bar{u}, P_0\bar{u}-\bar{u}_h)\leq &(\bar{y}-y_d, y_{P_0\bar{u},h}-y_{P_0\bar{u}})+\notag\\&(\bar{y}-y_d, y_{P_0\bar{u}}-y_{\bar{u}})+(\bar{y}-y_d, y_{\bar{u}_h}-y_{\bar{u}_h,h}).
\end{align}

Consider $(\bar{y}_h-\bar{y},y_{P_0\bar{u},h}-y_{\bar{u}_h,h})$ from \eqref{cntest2}.	
\begin{align}\label{cntest7}
(\bar{y}_h-\bar{y},y_{P_0\bar{u},h}-y_{\bar{u}_h,h})&=(y_{f,h}-y_f,y_{P_0\bar{u},h}-y_{\bar{u}_h,h})+(y_{\bar{u}_h,h}-y_{\bar{u}},y_{P_0\bar{u},h}-y_{\bar{u}_h,h})\notag\\
&\leq (y_{f,h}-y_f,y_{P_0\bar{u},h}-y_{\bar{u}_h,h})+(y_{P_0\bar{u},h}-y_{\bar{u}},y_{P_0\bar{u},h}-y_{\bar{u}_h,h}).
\end{align}	
From \eqref{cntest7} consider $\|y_{P_0\bar{u},h}-y_{\bar{u}_h,h}\|$.
\begin{align}\label{cntest8}
\|y_{P_0\bar{u},h}-y_{\bar{u}_h,h}\|\leq \|y_{P_0\bar{u},h}-y_{P_0\bar{u}}\|+\|y_{P_0\bar{u}}-y_{\bar{u}_h}\|+\|y_{\bar{u}_h}-y_{\bar{u}_h,h}\|
\end{align}	
From \eqref{cntest8} consider $\|y_{P_0\bar{u},h}-y_{P_0\bar{u}}\|.$	
\begin{align}\label{cntest9}
\|y_{P_0\bar{u},h}-y_{P_0\bar{u}}\|\leq \|y_{P_0\bar{u}}-y_{P_1(P_0\bar{u})}\|+\|y_{P_1(P_0\bar{u})}-y_{P_0\bar{u},h}\|.
\end{align}	
Combining \eqref{cntest9}, Theorem \ref{aux1}, Theorem \ref{supconv2} we obtain	
\begin{align}\label{cntest10}
\|y_{P_0\bar{u},h}-y_{P_0\bar{u}}\|\lesssim h\|\bar{u}\|_{H^{\frac{1}{2}}(\partial\Omega)}.
\end{align}	
Theorem \ref{sup-conv} yields,
\begin{align}\label{cntest11}
\|y_{\bar{u}_h,h}-y_{\bar{u}_h}\|\lesssim h.
\end{align}	
An application of ultra weak formulation yields that
\begin{align}\label{cntest11..5}
    \|y_{P_0\bar{u}}-y_{\bar{u}_h}\|\lesssim \|P_0\bar{u}-\bar{u}_h\|\lesssim C,
\end{align}
where $C>0$ is depends only upon $|u_a|,~|u_b|$, $|\Omega|$ and shape-regularity of the triangulation. Combining \eqref{cntest8}, \eqref{cntest9}, \eqref{cntest10}, \eqref{cntest11} and \eqref{cntest11..5} we obtain 
\begin{align}\label{cntest11..6}
    \|y_{P_0\bar{u},h}-y_{\bar{u}_h,h}\|\lesssim C
\end{align}
Therefore combining \eqref{cntest7} and \eqref{cntest11..6} one obtains
\begin{align}\label{cntest11..7}
    (\bar{y}_h-\bar{y},y_{P_0\bar{u},h}-y_{\bar{u}_h,h})&\leq(y_{f,h}-y_f,y_{P_0\bar{u},h}-y_{\bar{u}_h,h})+(y_{\bar{u}_h,h}-y_{\bar{u}},y_{P_0\bar{u},h}-y_{\bar{u}_h,h})\lesssim h,
\end{align}
where the generic constant is independent of the discretization parameter $h>0$ and depends only upon $\|\bar{u}\|_{H^{\frac{1}{2}}(\partial\Omega)}$, $\|f\|$, shape-regularity of the triangulation, $\Omega$.
Applying ultraweak formulation \eqref{ultra:weak} for $y_{\bar{u}-P_0\bar{u}}$ and duality argument we obtain 
\begin{align}\label{cntest12}
    \|y_{\bar{u}}-y_{P_0\bar{u}}\|\lesssim h\|\bar{u}\|_{H^{\frac{1}{2}}(\partial\Omega)}.
\end{align}
Hence combination of \eqref{cntest10}, \eqref{cntest12} and triangle inequality yields
\begin{align}\label{cntest13}
    \|y_{\bar{u}}-y_{P_0\bar{u},h}\|\lesssim h\|\bar{u}\|_{H^{\frac{1}{2}}(\partial\Omega)}.
\end{align}

Combining \eqref{cntest2}, \eqref{cntest3}, \eqref{cntest4}, \eqref{cntest7}, \eqref{cntest8}, \eqref{cntest9}, \eqref{cntest10}, \eqref{cntest11}, \eqref{cntest12}, \eqref{cntest13} we obtain 
\begin{align}
   \|\tilde{P}_1\bar{u}_h-\tilde{P}_1P_0\bar{u}\|\lesssim h^{1/2} \label{cntest14}
\end{align}
Note that
\begin{align}\label{cntest15}
    \|\bar{u}-\bar{u}_h\|\leq \|\bar{u}-P_0\bar{u}\|+\|P_0\bar{u}-\tilde{P}_1P_0\bar{u}\|+\|\tilde{P}_1\bar{u}_h-\tilde{P}_1P_0\bar{u}\|+\|\bar{u}_h-\tilde{P}_1\bar{u}_h\|.
\end{align}
Note that 
\begin{align}\label{cntest16}
    \|\bar{u}_h-\tilde{P}_1\bar{u}_h\|=\|P_0(\tilde{P}_1\bar{u}_h)-\tilde{P}_1\bar{u}_h\|\lesssim h^{1/2}\|\tilde{P}_1\bar{u}_h\|.
\end{align}
 We conclude our proof with the help of \eqref{cntest15}, \eqref{cntest14}, \eqref{cntest16} \eqref{l2est3} and \eqref{unifbdd3}.
 \end{proof}
\end{document}